\documentclass[11pt]{article}

\usepackage[centertags]{amsmath}
\usepackage{amsfonts}
\usepackage{amssymb}
\usepackage{amsthm}
\usepackage{newlfont}

\newtheorem{theorem}{Theorem}[section]

\theoremstyle{definition}

\newtheorem{thm}{Theorem}[section]
\newtheorem*{thm*}{Theorem}

\newtheorem{prop}[thm]{Proposition}
\newtheorem*{prop*}{Proposition}

\newtheorem*{dfn*}{Definition}

\theoremstyle{remark}
\newtheorem{rem}[theorem]{Remark}

\numberwithin{equation}{section}

\renewcommand{\labelenumi}{(\arabic{enumi})}

\newcommand{\abs}[1]{\left\vert#1\right\vert}
\newcommand{\set}[1]{\left\{#1\right\}}
\newcommand{\brac}[1]{\left(#1\right)}

\newcommand{\Real}{\mathbb{R}}

\newcommand{\eps}{\varepsilon}
\newcommand{\E}{\mathbb{E}}
\renewcommand{\S}{\mathbb{S}}

\newcommand{\I}{\mathcal{I}}

\renewcommand{\H}{\mathcal{H}}
\newcommand{\J}{\mathcal{J}}

\newlength{\defbaselineskip}
\newcommand{\setlinespacing}[1]%
           {\setlength{\baselineskip}{#1 \defbaselineskip}}

\begin{document}

\title{Isoperimetric Bounds on Convex Manifolds}

\author{Emanuel Milman\textsuperscript{1}}

\footnotetext[1]{Department of Mathematics, University of Toronto, 40 St. George Street, Toronto, Ontario M5S 2E4, Canada. Email: emilman@math.toronto.edu.
}


\maketitle

\begin{abstract}
We extend several Cheeger-type isoperimetric bounds for convex sets in Euclidean space, due to Bobkov and Kannan--Lov\'asz--Simonovits, to Riemannian manifolds having non-negative Ricci curvature. In order to extend Bobkov's bound, we require in addition an upper bound on the \emph{sectional} curvature of the space, which permits us to use comparison tools in Cartan--Alexandrov--Toponogov (or CAT) spaces. Along the way, we also
quantitatively improve our previous result that weak concentration assumptions imply a Cheeger-type isoperimetric bound, to a sharp bound with respect to all parameters.
\end{abstract}

\section{Introduction}

This work revolves around a new generalization to the Riemannian setting of the following theorem due to Sergey Bobkov \cite{BobkovVarianceBound} (we use the formulation from \cite{EMilman-RoleOfConvexity}, which is formally stronger but ultimately equivalent in the cases of interest):

\begin{thm*}[Bobkov]
Let $K$ denote a convex bounded domain in Euclidean space $(\Real^n,\abs{\cdot})$.
Let $X$ denote a random point uniformly distributed in $K$ (with respect to Lebesgue measure). Then for any $x_0 \in \Real^n$, denoting $E_{x_0} = \E(|X-x_0|)$ and $S_{x_0} = \S(|X-x_0|)$, we have:
\[
 D_{Che}(K) \geq \frac{c}{\sqrt{E_{x_0} S_{x_0}}} ~,
\]
for some universal constant $c>0$.
\end{thm*}

Let us explain the notation used above. We denote the expectation of a random variable $Y$ by $\E(Y)$, and set $\S(Y) := \sqrt{\E((Y - \E(Y))^2)}$ to denote the square root of the variance. We will use $D_{Che}(\Omega,d,\mu)$ to denote the Cheeger constant of the measure-metric space $(\Omega,d,\mu)$, where $d$ is a separable metric on $\Omega$ and $\mu$ is a Borel probability measure on $(\Omega,d)$. Given such a space, Minkowski's
(exterior) boundary measure of a Borel set $A \subset \Omega$, denoted $\mu^+(A)$, is
defined as $\mu^+(A) := \liminf_{\eps \to 0} \frac{\mu(A^d_{\eps}) -
\mu(A)}{\eps}$, where $A^d_{\eps} = A^{\Omega,d}_{\eps} := \set{x \in \Omega ; \exists y
\in A \;\; d(x,y) < \eps}$ denotes the $\eps$-neighborhood of $A$ in $(\Omega,d)$.
The isoperimetric profile $\I =
\I_{(\Omega,d,\mu)}$ is defined as the function $\I: [0,1] \rightarrow \Real_+$ given by $\I(v) = \inf \set{ \mu^+(A) ; \mu(A) = v}$. The Cheeger constant is then defined as:
\[
D_{Che}(\Omega,d,\mu) := \inf_{v \in [0,1]} \frac{\I_{(\Omega,d,\mu)}(v)}{\min(v,1-v)} = \inf_{A \subset \Omega} \frac{\mu^+(A)}{\min(\mu(A),1-\mu(A))} ~,
\]
measuring a certain linear isoperimetric property of the space $(\Omega,d,\mu)$.
When $d$ and $\mu$ are implied from the context, we simply write $D_{Che}(\Omega)$. So $D_{Che}(K)$ above denotes the Cheeger constant of $K$ with respect to the Euclidean metric $d = \abs{\cdot}$ and the uniform probability measure $\mu_K$ on $K$. Note that in this case, when $A \subset K$ has smooth boundary, then $\mu_K^+(A) = \H^{n-1}(\partial A \cap K) / \H^n(K)$, where $\H^k$ denotes the $k$-dimensional Hausdorff measure.

\medskip

Bobkov's proof in \cite{BobkovVarianceBound} is based on a geometric localization method,
systematically developed by Kannan, Lov\'asz and Simonovits \cite{LSLocalizationLemma,KLS}, and on a reverse H\"{o}lder inequality for $L_p$ norms of polynomials on convex sets \cite{Bourgain-LK,BobkovPolynomials,CarberyWrightPolynomials,NazarovSodinVolbergPolynomials}, which may also be proved by using the localization technique. A more geometric proof of Bobkov's theorem was given in \cite[Theorem 5.15]{EMilman-RoleOfConvexity}, relying on a general bound on $D_{Che}(K)$ obtained by Kannan--Lov\'asz--Simonovits \cite{KLS} using again localization.
The localization method has its origins in the work of Payne and Weinberger \cite{PayneWeinberger}, and was also rediscovered by Gromov and V. Milman \cite{Gromov-Milman}. The main drawback of this method is that, although being a very useful tool in the Euclidean setting, it can only be extended to very specific Riemannian manifolds (e.g. the Euclidean sphere, as in \cite{Gromov-Milman}).
This perhaps explains the difficulty in generalizing Bobkov's theorem to the more general Riemannian setting.

\medskip

Our main result generalizes Bobkov's theorem to the manifold setting, by employing the geometric method developed in \cite{EMilmanGeometricApproachPartI} and thereby avoiding the use of localization (see Remark \ref{rem:locVsJac} for a comparison with the latter method). In fact, we extend to this setting the Kannan--Lov\'asz--Simonovits result mentioned above as follows:

\begin{dfn*}
Given a bounded subset $\Omega$ of a complete Riemannian manifold $(M,g)$ and $x \in \Omega$, we denote by $\theta_\Omega(x)$ the Riemannian length of the longest minimizing geodesic contained in $\Omega$ and centered at $x$. In other words:
\[
 \theta_\Omega(x) := 2 \sup_{\xi \in S_x} \sup \set{t \geq 0 \; ; \begin{array}{c} \forall s \in [-t,t] \;\; \exp_x(s \xi) \in \Omega \text{ and } \\ d(\exp_x(-t\xi),\exp_x(t\xi)) = 2t \end{array} } ~,
\]
where $S_x$ denotes the unit sphere in the tangent space $T_x M$ and $d$ is the induced metric on $(M,g)$.
\end{dfn*}

\begin{thm}[generalized KLS bound] \label{thm:KLS-II}
Let $(M,g)$ denote a complete smooth oriented Riemannian manifold with non-negative Ricci curvature. Let $K$ denote a geodesically convex bounded domain in $(M,g)$ with (possibly empty) $C^2$ boundary,
and let $X$ denote a random point uniformly distributed in $K$ (with respect to the Riemannian volume form $vol_M$). Then:
\[
 D_{Che}(K,d,\mu_K) \geq \frac{c}{\E \theta_{K}(X)} ~,
\]
for some universal constant $c>0$.
\end{thm}
\noindent
This theorem in the Euclidean setting is due to Kannan--Lov\'asz--Simonovits \cite{KLS}.

\medskip

In order to control $\theta_K(x)$, we employ comparison results for geodesic triangles in Cartan--Alexandrov--Toponogov (or $CAT$) spaces, to which end we require in addition an \emph{upper} bound on the \emph{sectional} curvature. We formulate this in greater generality in Section \ref{sec:main2}, and state the result for the time being as follows (we employ throughout this work the convention that $1 / 0 = +\infty$):

\begin{thm}[Generalized Bobkov bound] \label{thm:main2}
Let $(M,g)$ denote a complete oriented smooth Riemannian manifold with non-negative Ricci curvature, let $d$ denote the induced geodesic distance, and let $\kappa \geq 0$ denote an upper bound on the sectional curvatures in $(M,g)$. Assume that one of the following assumptions holds:
\renewcommand{\labelenumi}{(\Alph{enumi})}
\begin{enumerate}
 \item The injectivity radius $\text{inj}(M)$ of $(M,g)$ is at least $\pi / \sqrt{\kappa}$ ; or,
\item The shortest geodesic loop in $(M,g)$ has length at least $2 \pi / \sqrt{\kappa}$ ; or,
\item $(M,g)$ is compact, even-dimensional, with strictly positive sectional curvatures.
\end{enumerate}
\renewcommand{\labelenumi}{(\arabic{enumi})}
Let $K$ denote a geodesically convex bounded domain of $(M,g)$ with (possibly empty) $C^2$ boundary, and let $X$ denote a random point uniformly distributed in $K$ (with respect to the Riemannian volume form $vol_M$). Given $x \in M$, denote $E_{x} = \E(d(X,x))$ and $S_{x} = \S(d(X,x))$, and set $R_x = E_x + 2 S_x$. Let $x_0 \in M$ be any point such that:
\begin{enumerate}
\item $R_{x_0} < \pi/(2\sqrt{\kappa})$ ;
\item Denoting $K_{x_0} = K \cap B(x_0,R_{x_0})$, there exists $\eps_0>0$ so that for any $\eps \in (0,\eps_0)$, there exists a domain $K_{x_0,\eps}$ so that:
\begin{enumerate}
 \item $K_{x_0} \subset K_{x_0,\eps} \subset (K_{x_0})^d_\eps$ ;
\item $\partial K_{x_0,\eps}$ is $C^{2}$ smooth ;
\item $K_{x_0,\eps}$ is geodesically convex.
\end{enumerate}
\end{enumerate}
Then:
\[
D_{Che}(K) \geq c \frac{\sqrt{1 - \frac{2}{\pi} E_{x_0} \sqrt{\kappa}} }{\sqrt{E_{x_0} S_{x_0}}} ~,
\]
for some universal constant $c>0$.
\end{thm}

\begin{rem}
As will be evident in Section \ref{sec:CAT}, under the upper bound assumption on the sectional curvature, assumption $(C)$ above implies assumption $(B)$ which implies assumption $(A)$; the latter assumption
implies in turn that $(M,g)$ is a $CAT(\kappa)$ space (see Theorem \ref{thm:CAT}). In such spaces, any open geodesic ball $B(x_0,R)$ with radius $R<\pi/(2 \sqrt{\kappa})$ is (strongly) geodesically convex (see e.g. \cite[II.1.4]{CAT0SpacesBook}
or the proof of Theorem \ref{thm:CAT}), implying in particular that $K_{x_0}$ is (strongly) geodesically convex for any $x_0 \in M$ satisfying condition (1) above. Consequently, condition (2) on $x_0$ is a mere technicality which we do not care to resolve here, and the reader is encouraged to morally disregard it.
\end{rem}

\begin{rem}
The constant $2$ in the definition $R_x = E_x + 2S_x$ is immaterial, and any other constant strictly greater than $1$ will yield the same result, with perhaps a different constant $c>0$ in the conclusion.
\end{rem}

Note that when $\kappa = 0$, we recover Bobkov's Theorem in the Euclidean setting (the case of $K$ with a non-smooth boundary may be obtained from the smooth boundary case by an approximation argument as in \cite[Section 6]{EMilman-RoleOfConvexity}). Although we have chosen to formulate for simplicity Theorem \ref{thm:main2} for the uniform probability measure on $K$ in a manifold with non-negative Ricci curvature, the same result holds for a more general probability measure of the form $\mu = \exp(-\psi) vol_M|_K$ with $Ric_g + D^2_g \psi \geq 0$ as $2$-tensors on $K$, where $Ric_g$ denotes the Ricci curvature tensor and $D_g$ denotes the covariant derivative on $(M,g)$ (we refer to \cite{EMilmanGeometricApproachPartI} for more details).

We also remark that since $1 - \frac{2}{\pi} E_{x_0} \sqrt{\kappa} > \frac{4}{\pi} S_{x_0} \sqrt{\kappa}$ according to condition (1), Theorem \ref{thm:main2} yields the looser bound:
\[
D_{Che}(K) \geq c_2 \frac{\kappa^{1/4}}{\sqrt{E_{x_0}}} \geq c_3 \sqrt{\kappa} ~,
\]
for any $x_0 \in M$ and $\kappa \geq 0$ satisfying the conditions of the theorem. It follows for instance that if $x_0 \in M$ satisfies:
\[
 E_{x_0} + 2 S_{x_0} < \min\brac{\frac{\pi}{2 \sqrt{\text{sec}(M)}} , \frac{\text{inj}(M)}{2}} ~,
\]
where $\text{sec}(M)$ denotes the supremum over the sectional curvatures of $M$, then:
\[
 D_{Che}(K) \geq \frac{c_4}{E_{x_0} + 2 S_{x_0}} \geq \frac{c_5}{E_{x_0}} ~;
\]
(see the proof of Theorem \ref{thm:main2'} for the last inequality). This should be compared to the general bound:
\[
 D_{Che}(K) \geq \sup_{x_0 \in M} \frac{c_5}{E_{x_0}} ~,
\]
obtained by Kannan--Lov\'asz--Simonovits for convex $K$ in the Euclidean setting, and extended to the Riemannian one in \cite{EMilman-RoleOfConvexity} assuming the Ricci curvature is non-negative, without any additional conditions.

\medskip

The rest of this work is organized as follows. In Section \ref{sec:warm}, we recall the geometric technique from \cite{EMilmanGeometricApproachPartI}, and quantitatively improve over \cite[Theorem 7.1]{EMilmanGeometricApproachPartI} using a new observation. Using it, together with a more refined analysis of the proof, we prove Theorem \ref{thm:KLS-II} in Section \ref{sec:KLS}. In Section \ref{sec:CAT}, we obtain several conditions for a domain in a Riemannian manifold to be a $CAT(\kappa)$ space, and derive a bound on $\theta_K$ in such spaces. Finally, in Section \ref{sec:main2}, we prove a more general version of Theorem \ref{thm:main2}.

\medskip

\noindent \textbf{Acknowledgements.} The author thanks the referee for useful comments that have improved the presentation of this work.


\section{Warm Up} \label{sec:warm}

Throughout this work, we denote by $(M,g)$ a complete oriented smooth Riemannian manifold, by $d$ the induced geodesic distance on $(M,g)$, and by $K$ a geodesically convex bounded domain in $(M,g)$ with (possibly empty) $C^2$ boundary. Recall that a subset $\Omega$ of $(M,g)$ is called geodesically convex if any two of its points may be connected by a minimizing geodesic which lies entirely inside $\Omega$; it is called strongly geodesically convex (or simply strongly convex) if any two of its points may be connected by a \emph{unique} minimizing geodesic in $M$, and this geodesic lies entirely inside $\Omega$. Clearly, the intersection of a geodesically convex domain and a strongly convex one is strongly convex. In addition, we will use $\mu_K$ to denote the uniform probability measure on $K$, i.e. $\mu_K = vol_M |_K / vol_M(K)$. Finally, given $p \in M$ and $R>0$, we denote by $B(p,R)$ the open geodesic ball of radius $R$ centered at $p$.

Another convention is that all constants denote some positive numeric values, independent of all other parameters (and in particular, independent of the dimension of $M$), whose value may change from one occurrence to the next.

The proof of Theorem \ref{thm:KLS-II} is an easy consequence of the method of proof of our previous result \cite[Theorem 7.1]{EMilmanGeometricApproachPartI}, which we can actually improve as follows:

\begin{thm}[Improved from \cite{EMilmanGeometricApproachPartI}] \label{thm:small-conc}
Assume that the $(M,g)$ has non-negative Ricci curvature, and that:
\begin{equation} \label{eq:conc-assump}
\exists \lambda_0 \in (0,1/2) \;\; \exists r_0 > 0 \;\; \text{such that} \;\; \mu_K(A) \geq 1/2 \Rightarrow 1-\mu_K(A^{d}_{r_0}) \leq \lambda_0 ~.
\end{equation}
Then:
\[
D_{Che}(K,d,\mu_K) \geq \frac{1-2\lambda_0}{r_0} ~.
\]
\end{thm}

The improvement lies in the quantitative dependence on the parameter $\lambda_0$ above. This dependence is now asymptotically best possible for values of $\lambda_0$ close to $1/2$ and $0$, as easily witnessed in the Euclidean setting when $K$ denotes the cube $[0,1]^n$ for the former possibility, and an extremely elongated cube $[0,1]^{n-1} \times [0,M]$ (with $M \rightarrow \infty$) for the latter one; in either case, the isoperimetric minimizing set having $\mu_K$-measure $1/2$ is known to be the half cube $[0,1]^{n-1} \times [0,M/2]$ (\cite{HadwigerCube}, see also \cite{BollobasLeader,BartheMaureyIsoperimetricInqs}), implying that $D_{Che}(K) = 2/M$. We also mention the alternative semi-group approach to the results of \cite{EMilmanGeometricApproachPartI} developed by M. Ledoux in \cite{LedouxConcentrationToIsoperimetryUsingSemiGroups}, which we do not see how to adapt to obtain the sharp estimate of Theorem \ref{thm:small-conc}.

\medskip

To obtain the required preparations for the proof of Theorem \ref{thm:KLS-II} and to demonstrate the aforementioned improvement, we first present the proof of Theorem \ref{thm:small-conc}. We refer to \cite{EMilmanGeometricApproachPartI} for a detailed account of all the tools and corresponding references used in the proof.

\begin{proof}[Proof of Theorem \ref{thm:small-conc}]
It is known that under our assumptions:
\begin{equation} \label{eq:half}
 D_{Che}(K,d,\mu_K) = 2 \I_{(K,d,\mu_K)}(1/2) = 2 \inf \set{ \mu_K^+(A) \; ; \; \mu_K(A) = 1/2} ~.
\end{equation}
This may be deduced as in \cite{EMilman-RoleOfConvexity} from the concavity of the isoperimetric profile $\I$ in this setting; an easier argument was later obtained in \cite{EMilmanGeometricApproachPartI}.

Let $A$ denote an open set $A$ with $\mu_K(A)=1/2$ and having minimal boundary measure $\mu_K^+(A)$ among all such sets. Geometric Measure Theory ensures that such an isoperimetric minimizer indeed exists, that its relative boundary $\partial A \cap K$ cannot have ``too many'' singularities, and provides a description of the structure of these possible singularities. In particular, the regular open subset of $\partial A \cap K$, which we denote by $\partial_r A$, is an analytic $n-1$ dimensional submanifold having a well defined unit outer normal, which we denote by $\nu_A(x)$ at $x \in \partial_r A$, and the singular part of $\overline{\partial A \cap K}$ is negligible enough to ensure that $\mu_K^+(A) = \H^{n-1}(\partial_r A) / vol_M(K)$.
In addition, this allows taking variations of $\partial_r A$, and it is classical that the minimality of the boundary measure implies that $\partial_r A$ must have constant mean curvature, which we denote by $H_A$.
Although this is not essential for the sequel, we mention that we use the following \emph{non-standard} convention for specifying the mean-curvature's sign: it is \emph{positive} for the sphere in Euclidean space with respect to the \emph{outer} normal. In addition, it is known that for any $x$ in the intersection of $\overline{\partial_r A}$ and $\partial K$, these two sets meet orthogonally at $x$, which together with the geodesic convexity of $K$ ensure that any closest point $P_x \in \overline{\partial A \cap K}$ to $x \in K$ is in fact in $\partial A \cap K$. Moreover, as first observed by Gromov \cite{GromovGeneralizationOfLevy}, $P_x$ will actually lie in $\partial_r A$, and so $x$ lies on the geodesic emanating from $P_x$ in the direction $\pm n_A(x)$. Consequently, the classical Heintze--Karcher comparison theorem in Riemannian Geometry \cite{HeintzeKarcher} (see also \cite{BuragoZalgallerBook}) applies to $\partial_r A$ as described next.

Define the normal exponential map $\exp_A: \partial_r A \times [0,\infty) \rightarrow M$ given by $\exp_A(x,s) = \exp_x(s \nu_A(x))$. It follows from the above discussion that the restriction of $\exp_A$ on $N_t := \partial_r A \times [0,t)$ is surjective on $A^{K,d}_t \setminus A$. Let $J_A(x,s)$ denote the Jacobian of this map, so that the pull-back of $vol_M$ by $\exp_A$ is given by $\exp_A^*(vol_{M}) = J_A(x,s) dvol_{\partial_r A}(x) ds$. As is well known (e.g. \cite{HeintzeKarcher,ChavelRiemannianGeometry1stEd}), $\exp_A$ remains onto $A^{K,d}_t \setminus A$ even when restricted to the set $N_t \cap DFoc(A)$, where:
\[
DFoc(A) := \set{ (x,s) \; ; \; x \in \partial_r A \;\; s \in [0,foc_A(x)] } ~,
\]
and $foc_A(x)$ denotes the first positive zero of $s \mapsto J_A(x,s)$ (and $+\infty$ if this function does not vanish on $\Real_+$); we will return to this point in greater detail during the proof of Theorem \ref{thm:KLS-II}. Consequently:
\begin{equation} \label{eq:loose}
 vol_M(K) \mu_K(A^d_t \setminus A) \leq \int_{\partial_r A} \int_0^{\min(t,foc_A(x))} J_A(x,s) ds \; dvol_{\partial_r A}(x) ~.
\end{equation}
The Heintze--Karcher theorem (in particular) bounds $J_A(x,\cdot)$ by the Jacobian $\J_{0,H_A}$ of the normal exponential map in the model space of constant curvature (determined by the lower bound $0$ on the Ricci curvature), corresponding to a surface at an umbilical point having the same mean-curvature as $\partial_r A$ at $x$, i.e. $H_A$. In our case, the model space is simply the Euclidean one, and so (e.g. \cite[Corollary 33.3.7]{BuragoZalgallerBook}):
\[
J_A(x,s) \leq \J_{0,H_A}(s) := (1 + H_A s)^{n-1} \;\;\;\; \forall s \in [0,foc_A(x)] ~.
\]
We remark that when $(M,g)$ is Euclidean, this is just a consequence of the Arithmetic-Geometric Means inequality, valid as long as $1 + k_{\min} s \geq 0$, where $k_{\min}$ denotes the smallest principle curvature of $\partial_r A$ at $x$. In any case, since $J_A(x,0) = 1$, we see that all terms above are non-negative in the specified range of $s$, by the definition of $foc_A(x)$.

The new observation is that by exchanging $A$ with $B = K \setminus A$ if necessary, we may always assume that $H_A \leq 0$ (note that $B$ is still an isoperimetric minimizer of measure $1/2$, since the regularity of $\partial A \cap K = \partial B \cap K$ ensures that $\mu_K^+(A) = \mu_K^+(B)$).
Consequently:
\[
 vol_M(K) \mu_K(A^d_t \setminus A) \leq t \H^{n-1}(\partial_r A) = t  \; vol_M(K) \mu_K^+(A) ~.
\]
Using (\ref{eq:half}), it follows that:
\begin{equation} \label{eq:small-conclusion}
 \frac{1}{2} D_{Che}(K,d,\mu_K) \geq \frac{\mu_K(A^d_t) - 1/2}{t} ~.
\end{equation}
The claimed result follows by using $t = r_0$ and our assumption that $\mu_K(A^d_{r_0}) \geq 1 - \lambda_0$.
\end{proof}


\section{Proof of Theorem \ref{thm:KLS-II}} \label{sec:KLS}

For the proof of Theorem \ref{thm:KLS-II}, we will need to be slightly more careful with the above argument, and exploit the Heintze--Karcher theorem in its full strength.

\begin{proof}[Proof of Theorem \ref{thm:KLS-II}]
We continue with the notations used in the proof of Theorem \ref{thm:small-conc}, and assume that $A$ is an open set with $\mu_K(A) = 1/2$ which minimizes $\mu_K^+(A)$ among all such sets, so that in addition $H_A \leq 0$.

Recall that the cut-distance function $cut_A : \partial_r A \rightarrow \Real_+ \cup \set{+\infty}$ is defined as follows:
\[
cut_A(x) = \sup \set{s > 0 \; ; \; d(\exp_x(s \nu_A(x)),\partial_r A) = s} ~.
\]
It is easy to verify (e.g. \cite[pp. 105,134]{ChavelRiemannianGeometry1stEd}) that the geodesic $s \mapsto \exp_x(s \nu_A(x))$, $s \in [0,t]$, is the unique geodesic minimizing distance between $\exp_x(t \nu_A(x))$ and $\partial_r A$ (equivalently $A$), for all $t \in (0,cut_A(x))$, and that it fails to minimize distance for all $t > cut_A(x)$. Let $Cut(A)$ denote the cut locus of $A$ in $M$, defined as:
\[
 Cut(A) = \set{\exp_x(cut_A(x) \nu_A(x)) \; ; \; x \in \partial_r A \;,\; cut_A(x) < \infty} ~,
\]
and set:
\[
 DCut(A) = \set{(x,s) \; ; \; x \in \partial_r A \;,\; s \in [0,cut_A(x)) } ~.
\]
It follows that restricting on $DCut(A)$, the normal exponential map $\exp_A$ is injective. It is also well known \cite{ChavelRiemannianGeometry1stEd} that the cut-distance cannot exceed the distance to the first focal point, i.e. that $cut_A(x) \leq foc_A(x)$, and so $\exp_A$ restricted on $DCut(A)$ is in fact a diffeomorphism onto its image. Combining with the discussion in the proof of Theorem \ref{thm:small-conc}, it follows that each $p \in K \setminus (A \cup Cut(A))$ may be reached by a unique geodesic $s \mapsto \exp_x(s \nu_A(x))$, $s \in [0,t]$, with $(x,t) \in DCut(A)$; this geodesic minimizes the distance to $\partial_r A$ and remains in $K$.

It is also well known \cite{ChavelRiemannianGeometry1stEd} that $cut_A: \partial_r A \rightarrow (0,+\infty]$ is a continuous function. Therefore the graph of $cut_A$ in the normal bundle over the $(n-1)$ dimensional open manifold $\partial_r A$ has zero $n$-dimensional Hausdorff measure, and consequently, so does its image under the (smooth) normal exponential map $\exp_A$, i.e. $vol_M(Cut(A)) = 0$ (see e.g. \cite{ChavelRiemannianGeometry1stEd}, and also \cite{JapaneseDistanceToCutLocus} for more precise information).
This implies that for any $f \in L_1(K \setminus A)$:
\begin{equation} \label{eq:integration}
 \int_{K \setminus A} f(x) dvol_M(x) = \int_{\partial_r A} \int_0^{cut^K_A(x)} J_A(x,s) f(\exp_x(s \nu_A(x))) \; ds  \; dvol_{\partial_r A}(x) ~,
\end{equation}
where $cut^K_A(x) = \min(cut_A(x),e_K(x))$, and $e_K(x) \in (0,+\infty]$ denotes the first time the geodesic $s \mapsto \exp_x(s \nu_A(x))$ exits $K$. Note that since $cut_A(x) \leq foc_A(x)$, $J_A(x,s)$ in the above integral is always positive (except possibly at the end point).

We would now like to obtain an upper bound for $E := \int d(x,A) d\mu_K(x)$. Using the above description, we obtain:
\[
 E = \frac{1}{vol_M(K)} \int_{\partial_r A} \int_0^{cut^K_A(x)} J_A(x,s) s \; ds \; dvol_{\partial_r A}(x) ~.
\]
We now employ the Heintze--Karcher theorem in its stronger version \cite{HeintzeKarcher} (see also \cite[Theorem 33.3.9]{BuragoZalgallerBook}), which states that:
\[
\frac{d}{ds} \log J_A(x,s) \leq \frac{d}{ds} \log \J_{0,H_A}(s) \;\;\; \forall s \in [0,foc_A(x)) ~.
\]
In particular, since $H_A \leq 0$, we observe that $J_A(x,\cdot)$ is non-increasing on $[0,cut^K_A(x)]$.
It is elementary to verify that for any non-increasing function $\varphi$ on $[0,e]$:
\[
 \int_0^e \varphi(s) s \; ds \leq 4 \int_0^{e/2} \varphi(s) s \; ds \leq 4 \int_0^{e} \varphi(s) \min(s,e-s) \; ds ~,
\]
and so we obtain:
\[
 E \leq \frac{4}{vol_M(K)} \int_{\partial_r A} \int_0^{cut^K_A(x)} J_A(x,s) \min(s,cut^K_A(x)-s) \; ds  \; dvol_{\partial_r A}(x) ~.
\]
Recall that the (half open) geodesic $\gamma_x : [0,cut^K_A(x)) \rightarrow M$ given by $s \mapsto \exp(s \nu_A(x))$, lies entirely inside $K$ and minimizes the distance between $\gamma_x(t)$ and $A$ for any $t \in [0,cut^K_A(x))$, and in particular, between $\gamma_x(t)$ and $x$. Consequently:
\[
2 \min(s,cut^K_A(x)-s) \leq \theta_K(\exp_x(s \nu_A(x))) \;\;\; \forall s \in [0,cut^K_A(x)) ~.
\]
Using (\ref{eq:integration}) again (a mere inequality as in (\ref{eq:loose}) would not suffice), we therefore obtain:
\begin{equation} \label{eq:E-estimate}
 E \leq \frac{2}{vol_M(K)} \int_{K \setminus A} \theta_K(x) dvol_M(x) \leq 2 \int \theta_K(x) d\mu_K(x) ~.
\end{equation}

We can now conclude as in the proof of Theorem \ref{thm:small-conc}. Choosing $r_0 = 4 E$, we obtain by the Markov--Chebyshev inequality that:
\[
1 - \mu_K(A^d_{r_0}) = \mu_K(\set{d(x,A) \geq r_0}) \leq \frac{E}{r_0} = \frac{1}{4} ~,
\]
and using (\ref{eq:small-conclusion}), we conclude that $D_{Che}(K,d,\mu_k) \geq 1/(2 r_0) = 1/(8 E)$. Recalling the estimate (\ref{eq:E-estimate}), the conclusion of the theorem follows with $c = 1/16$.
\end{proof}

\begin{rem} \label{rem:locVsJac}
In the Euclidean setting, the localization method mentioned in the Introduction reduces the study of certain ensembles of inequalities in $\Real^n$, to the study of the same inequalities on a ``needle'', i.e. a one-dimensional interval endowed with a measure of the form $l(s)^{n-1} ds$, where $l$ is a concave function (and in fact, a further reduction shows that $l$ may be assumed affine). The analogy with our method becomes apparent, after observing (e.g. \cite[Lemma 4.20]{GHLBook}) that $s \mapsto J_A(x,s)^{1/(n-1)}$ is always a concave function when the Ricci curvature is non-negative along the geodesic $s \mapsto \exp_x(s \nu_A(x))$. So our method may be thought of as integration over all of these ``needles'' emanating from $\partial_r A$, and provides a concrete geometric interpretation of the localization method, which is also valid in the Riemannian setting.
\end{rem}


\section{Estimating $\theta_K$ in $CAT(\kappa)$ spaces} \label{sec:CAT}

To use the estimate given by Theorem \ref{thm:KLS-II}, we will use the following theorem, which follows from the work of Cartan, Alexandrov and Toponogov, together with some other results in Riemannian Geometry.

\begin{thm} \label{thm:CAT}
Let $\Omega$ denote a geodesically convex domain in a smooth oriented Riemannian manifold $(M,g)$ (having induced metric $d$), and let $\kappa \in \Real$ denote an upper bound for the sectional curvatures of $(M,g)$ in $\Omega$. Assume in addition that one of the following conditions holds:
\renewcommand{\labelenumi}{(\Alph{enumi})}
\begin{enumerate}
\item $\Omega$ is strongly convex in $(M,g)$.
\item The injectivity radius of $\Omega$ is at least $\pi / \sqrt{\max(\kappa,0)}$.
\item The shortest geodesic loop in $(M,g)$ has length at least $2 \pi / \sqrt{\max(\kappa,0)}$.
\item $(M,g)$ is compact, even-dimensional, with strictly positive sectional curvatures bounded above by $\kappa>0$.
\item $(M,g)$ is simply-connected with sectional curvatures bounded above by $\kappa \leq 0$.
\end{enumerate}
\renewcommand{\labelenumi}{(\arabic{enumi})}
Then $(\Omega,d)$ is a $CAT(\kappa)$ space.
\end{thm}

Let us recall the definition of a $CAT(\kappa)$ space \cite{CAT0SpacesBook}. Denote by $M_\kappa$ ($\kappa \in \Real$) the simply connected two-dimensional model space of constant sectional curvature $\kappa$, let $d_\kappa$ denote the induced metric on $M_\kappa$, and let $D_\kappa$ denote its diameter, i.e. $D_k = \pi/\sqrt{\kappa}$ if $\kappa>0$ and $+\infty$ otherwise. Let $(\Omega,d)$ denote a general metric space. A continuous mapping $\gamma: [a,b] \rightarrow (\Omega,d)$ is called a geodesic if its length $L(\gamma)$, defined as:
\[
 L(\gamma) := \sup \set{ \sum_{i=0}^{n-1} d(\gamma(t_i),\gamma(t_{i+1})) \; ; \; a = t_0 < t_1 < \ldots < t_{n-1} < t_n = b } ~,
\]
is equal to $d(\gamma(a),\gamma(b))$. This coincides with the definition of a minimizing geodesic on a Riemannian manifold, so we will henceforth refer to $\gamma$ as a minimizing geodesic, or (minimizing geodesic) segment $[\gamma(a),\gamma(b)]$. Let $\Delta = \Delta(p,q,r)$ be a geodesic triangle, meaning that its sides are minimizing geodesic segments $[p,q]$,$[q,r]$,$[r,p]$ connecting their corresponding endpoints, and let $L(\Delta) = d(p,q) + d(q,r) + d(p,r)$ denote its perimeter. It is known that if $L(\Delta) < 2 D_\kappa$, then there exists a comparison geodesic triangle $\Delta_0 = \Delta(p_0,q_0,r_0)$ in $M_\kappa$ having the same side lengths (which is also unique up to isometry). Let us write $y \in \Delta$ to denote that $y$ lies on one of the sides of $\Delta$. A point $x_0 \in \Delta_0$ is called a comparison point for $x \in \Delta$ if its distance from the endpoints of the side it lies on (say $p_0,q_0$) is exactly equal to that of $x$ (i.e. $d_\kappa(x_0,p_0) = d(x,p)$, and hence $d_\kappa(x_0,q_0) = d(x,q)$). The space $(\Omega,d)$ is called a $CAT(\kappa)$ space if the following two conditions hold:
\begin{enumerate}
\item Any two points $x,y \in \Omega$ with $d(x,y) < D_{\kappa}$ may be joined by a minimizing geodesic.
\item For any geodesic triangle $\Delta = \Delta(p,q,r)$ in $(\Omega,d)$ with perimeter $L(\Delta) < 2 D_{\kappa}$, and any two points $x,y \in \Delta$, the corresponding comparison points $x_0,y_0$ on the comparison triangle $\Delta_0$ in $M_\kappa$ satisfy $d(x,y) \leq d_\kappa(x_0,y_0)$.
\end{enumerate}
It was shown by Alexandrov \cite{AlexandrovCATPaper} that for a metric space satisfying condition (1), condition (2) is equivalent to the following condition:
\renewcommand{\labelenumi}{(\arabic{enumi}')}
\begin{enumerate}
\setcounter{enumi}{1}
\item For any geodesic triangle $\Delta = \Delta(p,q,r)$ in $(\Omega,d)$ with perimeter $L(\Delta) < 2 D_{\kappa}$,
with $p \neq q$ and $p \neq r$, if $\gamma$ denotes the Alexandrov angle between $[p,q]$ and $[p,r]$ at $p$, and if $\Delta_0 = \Delta(p_0,q_0,r_0)$ is a geodesic triangle in $M_\kappa$ with $d_{\kappa}(p_0,q_0) = d(p,q)$, $d_{\kappa}(p_0,r_0) = d(p,r)$ and angle at $p_0$ equal to $\gamma$, then $d(q,r) \geq d_{\kappa}(q_0,r_0)$.
\end{enumerate}
\renewcommand{\labelenumi}{(\arabic{enumi})}
For a definition of the Alexandrov angle we refer to \cite{CAT0SpacesBook}, and only remark that on Riemannian manifolds, the Alexandrov angle coincides with the usual Riemannian angle between geodesics emanating from a common point.

\medskip

Since we could not find a reference for our desired formulation of Theorem \ref{thm:CAT}, we provide a proof for completeness:

\begin{proof}[Proof of Theorem \ref{thm:CAT}]
Condition (1) is satisfied since $\Omega$ is geodesically convex. Condition (2') may be deduced as a consequence of Rauch's Comparison Theorem; let us briefly sketch how this may be done, paying in particular attention to the reason behind our additional assumptions on $\Omega$ and $M$.

Given a geodesic triangle $\Delta = \Delta(p,q,r)$ in $\Omega$ with $L(\Delta)<2 D_\kappa$, it easily follows from the triangle inequality that $\Delta \subset B(p, D_\kappa)$.
We would like to lift via $\exp_p^{-1}$ the minimizing geodesic $[q,r]$ to a smooth curve in $T_p M$; to this end, we need to assume that a neighborhood of $[q,r]$ is disjoint from the cut-locus $Cut_p$ of $p$, so that $\exp_p$ is a diffeomorphism onto its image when restricting to the corresponding domain in $T_p M$. Our additional assumptions take care of this point:
\renewcommand{\labelenumi}{(\Alph{enumi})}
\begin{enumerate}
 \item
If $\Omega$ is strongly convex, then for any $p \in \Omega$, $Cut_p$ is disjoint from $\Omega$ by the uniqueness of the minimizing geodesic (e.g. \cite[Theorem 2.1.14]{KlingenbergRiemannianGeometry1stEd}).
\item
If the injectivity radius of $p$ in $\Omega$ is at least $D_\kappa$, which is just another way of saying that $d(p,Cut_p \cap \Omega) \geq D_\kappa$, then we are also in the clear since $\Delta \subset B(p,D_\kappa)$.
\item
We note that our curvature assumption guarantees by the Morse--Schoenberg Theorem (e.g.
\cite[Theorem 2.14]{ChavelRiemannianGeometry1stEd}) that there are no conjugate points to $p$ on a minimizing geodesic in $\Omega$ connecting $p$ to any $y \in B(p,D_{\kappa}) \cap \Omega$, so we just need to make sure that there are no cut-points of $p$ which are non-conjugate points in $B(p,D_{\kappa}) \cap \Omega$. Our third assumption ensures by Klingenberg's Lemma (e.g. \cite[Theorem 3.4]{ChavelRiemannianGeometry1stEd})
that there are no such points in $B(p,D_{\kappa})$ for any $p \in M$.
\item
By the theorems of Synge \cite[Theorem 2.6.7 (ii)]{KlingenbergRiemannianGeometry1stEd} and Klingenberg \cite[Theorem 2.6.9]{KlingenbergRiemannianGeometry1stEd}, the fourth assumption implies the third one.
\item
Our last assumption is classical: it implies by the Hadamard--Cartan theorem \cite[2.6.6]{KlingenbergRiemannianGeometry1stEd} that the injectivity radius is infinite, i.e. that $\exp_p : T_p M \rightarrow M$ is a diffeomorphism for any $p\in M$.
\end{enumerate}
\renewcommand{\labelenumi}{(\arabic{enumi})}
We conclude that under our assumptions, a smooth lift of $[q,r]$ into $T_p M$ is feasible. By comparing to the model space $M_\kappa$, the rest of the proof follows from Rauch's Theorem as in \cite[Theorem 2.7.6]{KlingenbergRiemannianGeometry1stEd}.
\end{proof}

In a $CAT(\kappa)$ space, it is easy to control the length of the longest symmetric minimizing geodesic inside an annulus, as shown in the following:
\begin{prop} \label{prop:CAT}
Let $p \in (M,g)$, and let $\Omega$ denote a geodesically convex open subset of $B(p, R_2)$. Assume that $(\Omega,d)$ is a $CAT(\kappa)$ space for some $\kappa \geq 0$,
and if $\kappa > 0$ assume that $R_2 \leq \pi/(2 \sqrt{\kappa})$.
Then for any $x \in \Omega$ with $d(x,p) \geq R_1 \geq 0$, we have the following estimate:
\[
\theta_\Omega(x) \leq \begin{cases} \frac{2}{\sqrt{\kappa}} \cos^{-1}\brac{ \frac{\cos(R_2 \sqrt{\kappa})}{\cos(R_1 \sqrt{\kappa})}} & \kappa > 0 \\ 2 \sqrt{R_2^2 - R_1^2} & \kappa = 0 \end{cases} ~.
\]
\end{prop}
We remark that, as follows from the proof, this estimate is sharp when $\Omega = B(p,R_2)$ in $M_\kappa$ (for $R_2$ in the above range), and in particular, the bound for $\kappa=0$ is obvious in the Euclidean case. An analogous bound may be obtained for the case $\kappa < 0$.

\begin{proof}
Given $x \in \Omega$ with $d(x,p) \geq R_1$, let $[q,r]$ denote a minimizing geodesic connecting $q,r \in \Omega$ and centered at $x$, and consider the geodesic triangle $\Delta = \Delta(p,q,r)$. Since $d(p,q) < R_2$, $d(p,r) < R_2$ and $d(q,r) < 2R_2$, it follows that $L(\Delta) < 4R_2 \leq 2 D_\kappa$. Let $\Delta_0 = \Delta(p_0,q_0,r_0)$ denote the comparison triangle to $\Delta$ in $M_{\kappa}$, and let $x_0$ denote the comparison point to $x$, i.e. the mid point of $[q_0,r_0]$. The $CAT(\kappa)$ property yields the second inequality below:
\begin{equation} \label{eq:triangle1}
 R_1 \leq d(p,x) \leq d(p_0,x_0) ~.
\end{equation}
Consider the angle $\gamma_{p_0,q_0}$ at $x_0$ between the minimizing geodesic segments $[x_0,p_0]$ and $[x_0,q_0]$. Replacing $q_0$ by $r_0$ if necessary, we may assume that $\gamma_{p_0,q_0} \geq \pi/2$. Denote by $L_{p_0,x_0},L_{x_0,q_0},L_{q_0,p_0}$ the side lengths of the geodesic triangle $\Delta'_0 = \Delta(p_0,x_0,q_0)$, and note that they are all smaller than $R_2$. We proceed assuming $\kappa > 0$, the case $\kappa = 0$ follows similarly (and more easily). By the spherical law of cosines \cite{CAT0SpacesBook}, we know that:
\[
 \cos(L_{p_0,q_0} \sqrt{\kappa}) = \cos(L_{p_0,x_0} \sqrt{\kappa}) \cos(L_{q_0,x_0} \sqrt{\kappa}) +
\sin(L_{p_0,x_0} \sqrt{\kappa}) \sin(L_{q_0,x_0} \sqrt{\kappa}) \cos(\gamma_{p_0,q_0}) ~.
\]
Since $R_2 \sqrt{\kappa} \leq \pi/2$ and $\gamma_{p_0,q_0} \geq \pi/2$, the rightmost term is non-positive, and together with (\ref{eq:triangle1}), it follows that:
\[
 \cos(R_2 \sqrt{\kappa}) \leq \cos(R_1 \sqrt{\kappa}) \cos(L_{q_0,x_0} \sqrt{\kappa}) ~.
\]
Solving for $L_{q_0,x_0}$, and since $d(q,r) = 2 L_{q_0,x_0}$, the asserted bound on $\theta_\Omega(x)$ follows.
\end{proof}


\section{Proof of Theorem \ref{thm:main2}} \label{sec:main2}

We are now ready to provide a proof of Theorem \ref{thm:main2}, which we state here in greater generality. The idea is the same one we used in \cite{EMilman-RoleOfConvexity} to re-derive Bobkov's Theorem in Euclidean space.

\begin{thm} \label{thm:main2'}
Let $(M,g)$ denote a complete oriented smooth Riemannian manifold with non-negative Ricci curvature, and let $d$ denote the induced geodesic distance. Let $K$ denote a geodesically convex bounded domain of $(M,g)$ with (possibly empty) $C^2$ boundary, and let $X$ denote a random point uniformly distributed in $K$ (with respect to the Riemannian volume form $vol_M$). Given $x \in M$, denote $E_{x} = \E(d(X,x))$, $S_{x} = \S(d(X,x))$, $R_x = E_x + 2 S_x$, and set $K_{x} := K \cap B(x,R_x)$.
Let $x_0 \in M$ and $\kappa_{x_0} \geq 0$ satisfy:
\begin{enumerate}
\item
$\kappa_{x_0}$ is an upper bound on the sectional curvatures of $(M,g)$ in $K_{x_0}$ ;
\item
$R_{x_0} < \pi/(2\sqrt{\kappa_{x_0}})$ ;
\item
There exists $\eps_0>0$ so that for any $\eps \in (0,\eps_0)$, there exists a domain $K_{x_0,\eps}$ so that:
\begin{enumerate}
\item $K_{x_0} \subset K_{x_0,\eps} \subset (K_{x_0})^d_\eps$ ;
\item $\partial K_{x_0,\eps}$ is $C^{2}$ smooth ;
\item $K_{x_0,\eps}$ is geodesically convex ;
\item
One of the following assumptions holds:
\begin{enumerate}
\item $K_{x_0,\eps}$ is in fact strongly geodesically convex ; or,
\item The injectivity radius of $K_{x_0,\eps}$ is at least $\pi / \sqrt{\kappa_{x_0}}$ ; or,
\item The shortest geodesic loop in $(M,g)$ has length at least $2 \pi / \sqrt{\kappa_{x_0}}$ ; or,
\item $(M,g)$ is compact, even-dimensional, with strictly positive sectional curvatures bounded above by $\kappa_{x_0}$.
\end{enumerate}
\end{enumerate}
\end{enumerate}
Then:
\[
D_{Che}(K) \geq c \frac{\sqrt{1 - \frac{2}{\pi} E_{x_0} \sqrt{\kappa_{x_0}}}}{\sqrt{E_{x_0} S_{x_0}}} ~,
\]
for some universal constant $c>0$.
\end{thm}
\noindent As already mentioned, the distinction between $K_{x_0}$ and $K_{x_0,\eps}$ is for technical reasons which we do not care to resolve here, and the reader should morally disregard this distinction.
\begin{proof}
Given a Borel set $\Omega \subset (M,g)$ with $vol_M(\Omega)>0$, denote $\mu_{\Omega} = vol_M|_{\Omega} / vol_M(\Omega)$. Let $x_0 \in M$ satisfy the above assumptions, and note that by Chebyshev's inequality, $\mu_K(K_{x_0}) \geq 3/4$, and hence:
\begin{equation} \label{eq:goal2}
D_{Che}(K,d,\mu_K) \geq D_{Che}(K_{x_0},d,\mu_{K_{x_0}}) /2 ~;
\end{equation}
this follows by the argument described in the proof of \cite[Lemma 5.2]{EMilman-RoleOfConvexity}, which only uses the fact that $D_{Che}(K,d,\mu_K) = 2 \I_{(K,d,\mu_K)}(1/2)$ in our setting (as explained in the proof of Theorem \ref{thm:small-conc}). So it remains to bound $D_{Che}(K_{x_0},d,\mu_{K_{x_0}})$. Recall that for small enough $\eps>0$, $K_{x_0,\eps}$ was assumed geodesically convex and having $C^2$ smooth boundary. By an approximation argument described in \cite[Theorem 6.10]{EMilman-RoleOfConvexity},
$D_{Che}(K_{x_0},d,\mu_{K_{x_0}}) = \lim_{\eps \rightarrow 0} D_{Che}(K_{x_0,\eps},d,\mu_{K_{x_0,\eps}})$, so it is enough to bound the latter expression.

Consider first the case that $E_{x_0} < 2 S_{x_0}$. Then $K_{x_0}$, and consequently $K_{x_0,\eps}$ for small enough $\eps>0$, have diameter smaller than $8S_{x_0}$.
Applying e.g. Theorem \ref{thm:small-conc} to $K_{x_0,\eps}$ with $r_0 = 8S_{x_0}$ and $\lambda_0 = 0$, and passing to the limit as $\eps\rightarrow 0$, we conclude that $D_{Che}(K_{x_0},d,\mu_{K_{x_0}}) \geq 1/(8S_{x_0})$. If $X_0$ denotes a random point chosen according to the uniform distribution in $K_{x_0}$, then an analogue of Borell's lemma \cite{Borell-logconcave} in our setting (see e.g. \cite[Lemma 6.13]{EMilman-RoleOfConvexity}) implies that $S_{x_0} = \sqrt{\E(X_0^2) - \E(X_0)^2} \leq \sqrt{\E(X_0^2)} \leq C E_{x_0}$, for some universal constant $C>0$. This confirms the asserted bound in the case that $E_{x_0} < 2 S_{x_0}$.

We proceed assuming that $E_{x_0} \geq 2 S_{x_0}$.
It is clear by compactness of $\overline{K_{x_0}}$ that we can find an upper bound $\kappa_{x_0,\eps} \geq \kappa_{x_0}$ on the sectional curvatures of $(M,g)$ in $K_{x_0,\eps}$, so that $\lim_{\eps \rightarrow 0} \kappa_{x_0,\eps} = \kappa_{x_0}$, and hence for small enough $\eps>0$:
\[
 R_{x_0,\eps} := R_{x_0} + \eps \leq \frac{\pi}{2 \sqrt{\kappa_{x_0,\eps}}} ~.
\]
Consequently, the validity of one of the assumptions (i),(ii),(iii),(iv) above ensures by Theorem \ref{thm:CAT} that for this range of $\eps>0$, $(K_{x_0,\eps},d)$ is a $CAT(\kappa_{x_0,\eps})$ space. Combining the bound obtained in Theorem \ref{thm:KLS-II} with the bound on $\theta_{K_{x_0,\eps}}$ given by Proposition \ref{prop:CAT}, we obtain:
\begin{eqnarray*}
 D_{Che}(K_{x_0,\eps},d,\mu_{K_{x_0,\eps}}) & \geq & \frac{c}{\int \theta_{K_{x_0,\eps}}(x) d\mu_{K_{x_0,\eps}}(x)} \\
&\geq & \frac{c}{\int \frac{2}{\sqrt{\kappa_{x_0,\eps}}} \cos^{-1}\brac{ \frac{\cos(R_{x_0,\eps} \sqrt{\kappa_{x_0,\eps}})}{\cos(d(x,x_0) \sqrt{\kappa_{x_0,\eps}})}} d\mu_{K_{x_0,\eps}}(x)} ~.
\end{eqnarray*}
Taking the limit as $\eps \rightarrow 0$ we conclude that:
\begin{equation} \label{eq:goal}
 D_{Che}(K_{x_0},d,\mu_{K_{x_0}}) \geq \frac{\frac{c}{2} \sqrt{\kappa_{x_0}}}{\int \cos^{-1}\brac{ \frac{\cos(R_{x_0} \sqrt{\kappa_{x_0}})}{\cos(d(x,x_0) \sqrt{\kappa_{x_0}})}} d\mu_{K_{x_0}}(x)} ~.
\end{equation}

It remains to evaluate the integral appearing in the denominator. Denote for convenience $\kappa=\kappa_{x_0}$, $E=E_{x_0}$, $S=S_{x_0}$, $R_+ = R_{x_0}$ and $R_{-} = E - 2S$. Integrating by parts, we obtain:
\begin{multline*}
\int_0^{R_+} \cos^{-1}\brac{ \frac{\cos(R_+ \sqrt{\kappa})}{\cos(t \sqrt{\kappa})}} d (\mu_{K_{x_0}}\set{d(x_0,x)\leq t}) \\
\leq \cos^{-1}\brac{ \frac{\cos( R_+ \sqrt{\kappa} )}{\cos(R_{-} \sqrt{\kappa})}} +\int_0^{R_-} \frac{\sqrt{\kappa} \cos(R_+ \sqrt{\kappa})\tan(t \sqrt{\kappa})}{\sqrt{\cos^2(t \sqrt{\kappa}) - \cos^2(R_+ \sqrt{\kappa})}} \mu_{K_{x_0}}\set{d(x_0,x)\leq t} dt ~.
\end{multline*}
Denote the first and second terms on the right-hand side above by $A$ and $B$ respectively.

Using that there exists a constant $C>1$ so that $\cos^{-1}(x) \leq C \sqrt{1-x}$ for all $x \in [0,1]$, and applying the mean-value theorem, we bound:
\[
A \leq  C \sqrt{\frac{\cos(R_{-} \sqrt{\kappa}) - \cos(R_{+} \sqrt{\kappa})}{\cos(R_{-} \sqrt{\kappa})}} \leq  C \sqrt{ \frac{4 S \sqrt{\kappa} \sin(R_+ \sqrt{\kappa})}{\cos(E\sqrt{\kappa})}} \leq C' \frac{ \sqrt{\kappa} \sqrt{S E} }{\sqrt{1 - \frac{2}{\pi} E \sqrt{\kappa}}} ~.
\]
Using Chebyshev's inequality (for $t < R_+$):
\[
\mu_{K_{x_0}}\set{d(x_0,x)\leq t} = \frac{\mu_K \set{d(x_0,x)\leq t}}{\mu_K(K_{x_0})}
\leq \frac{4}{3} \frac{S^2}{(E-t)^2} ~,
\]
we verify a similar bound on $B$:
\begin{eqnarray*}
\frac{3}{4} B & \leq & \int_0^{R_-} \frac{\sqrt{\kappa} \cos(R_+ \sqrt{\kappa})\tan(t \sqrt{\kappa})}{\sqrt{\cos^2(t \sqrt{\kappa}) - \cos^2(R_+ \sqrt{\kappa})}} \frac{S^2}{(E-t)^2} dt \\
&\leq & \frac{\sqrt{\kappa} \cos(R_+ \sqrt{\kappa})\tan(R_{-} \sqrt{\kappa})}{\sqrt{\cos^2(R_{-} \sqrt{\kappa}) - \cos^2(R_+ \sqrt{\kappa})}} \int_{2S}^\infty \frac{S^2}{t^2} dt\\
& \leq & \frac{ \sqrt{\kappa} \cos(R_+ \sqrt{\kappa})\sin(R_{-} \sqrt{\kappa})}{\sqrt{4 S \sqrt{\kappa} \sin(R_{-} \sqrt{\kappa}) }  \cos^{3/2}(R_{-} \sqrt{\kappa}) } \frac{S}{2} \\
& \leq & \frac{\sqrt{\kappa} \sqrt{S} \sqrt{R_{-}}}{4 \sqrt{\cos(E \sqrt{\kappa})}} \leq  \frac{\sqrt{\kappa} \sqrt{S E}}{4 \sqrt{1 - \frac{2}{\pi} E \sqrt{\kappa}}} ~.
\end{eqnarray*}
Plugging these estimates into (\ref{eq:goal}) and (\ref{eq:goal2}), the assertion of the theorem follows.
\end{proof}

\begin{rem}
It is actually possible to replace the requirements (a) and (b) from $K_{x_0,\eps}$ by the requirements that:
\renewcommand{\labelenumi}{(\alph{enumi}')}
\begin{enumerate}
\item The volume of the symmetric difference of $K_{x_0,\eps}$ and $K_{x_0}$ tends to $0$ as $\eps \rightarrow 0$ ;
\item $\partial K_{x_0,\eps}$ is only $C^{1,1}$ smooth ;
\end{enumerate}
\renewcommand{\labelenumi}{(\arabic{enumi})}
we omit the details. In particular, we believe that if $K_{x_0}$ is (strongly) geodesically convex and is bounded away from its cut-locus, taking $K_{x_0,\eps} := (K_{x_0})^d_\eps$, for which (a') and (b') hold, should still be (strongly) geodesically convex, allowing a simplification in the formulation of Theorems \ref{thm:main2} and \ref{thm:main2'}.
\end{rem}

\bibliographystyle{amsplain}

\def\cprime{$'$}
\providecommand{\bysame}{\leavevmode\hbox to3em{\hrulefill}\thinspace}
\providecommand{\MR}{\relax\ifhmode\unskip\space\fi MR }
\providecommand{\MRhref}[2]{%
  \href{http://www.ams.org/mathscinet-getitem?mr=#1}{#2}
}
\providecommand{\href}[2]{#2}

\end{document}